\documentclass[11pt]{amsart}

\usepackage{amsmath,amssymb,latexsym,soul,cite,mathrsfs}

\usepackage{color,enumitem,graphicx}
\usepackage[colorlinks=true,urlcolor=blue,
citecolor=red,linkcolor=blue,linktocpage,pdfpagelabels,
bookmarksnumbered,bookmarksopen]{hyperref}
\usepackage[english]{babel}

\usepackage[left=2.9cm,right=2.9cm,top=2.8cm,bottom=2.8cm]{geometry}
\usepackage[hyperpageref]{backref}

\usepackage[colorinlistoftodos]{todonotes}
\makeatletter
\providecommand\@dotsep{5}
\def\listtodoname{List of Todos}
\def\listoftodos{\@starttoc{tdo}\listtodoname}
\makeatother

\numberwithin{equation}{section}

\newtheorem{theorem}{Theorem}[section]

\newtheorem{lemma}[theorem]{Lemma}

\newtheorem{remark}{Remark}
\newtheorem{definition}{Definition}[section]

\begin{document}

\title[Some qualitative properties for ... ]
{Some qualitative properties for the Kirchhoff total variation flow}
\author{ Tahir Boudjeriou}
\address[Tahir Boudjeriou]
{\newline\indent  
	Department of Mathematics
	\newline\indent Faculty of Exact Sciences
	\newline\indent Lab. of Applied Mathematics
	\newline\indent University of Bejaia, Bejaia, 06000, Algeria 
	\newline\indent
	e-mail:re.tahar@yahoo.com}

\pretolerance10000

\begin{abstract}
 In this paper we are concerned with the following Kirchhoff type problem involving the 1-Laplace operator :
\begin{equation*}
\left\{\begin{array}{llc}
u_{t}-m\left(\int_{\Omega}|Du|\right)\Delta_{1} u=0 & \text{in}\ & \Omega\times (0,+\infty) , \\
u=0 & \text{on} &\partial \Omega\times (0,+\infty),\\
u(x,0)=u_{0}(x) & \text{in} &\Omega , 
\end{array}\right.
\end{equation*}
where $\Omega\subset \mathbb{R}^{N}$ ($N\geq 1$) is a bounded smooth domain, $m :\mathbb{R}_{+}\rightarrow \mathbb{R}_{+}$ is an increasing continuous function that satisfies some conditions which will be mentioned further down, and $\Delta_1 u=\text{div}\left(\frac{Du}{|Du|}\right)$ denotes the 1-Laplace operator. The main purpose of this work is to investigate from the initial data $u_{0}$ and the nonlinear function $m$ the existence and asymptotic behavior of solutions near the extinction time.
 \end{abstract}
\thanks{}
\subjclass[2000]{35K55, 35B40} 
\keywords{Kirchhoff Total variation 
flow, extinction time}

\maketitle	
\section{Introduction and the main results} 
In recent years, Andreu, Ballester, Caselles and Maz\'on \cite{Andreu}, and Andreu,  Caselles,  D\'iaz and Maz\'on \cite{Andreu2} studied the following problem of total variation flow
\begin{equation}\label{P}
\left\{\begin{array}{llc}
u_{t}-\text{div}\left(\frac{Du}{|Du|}\right)=0 & \text{in}\ & \Omega\times (0,+\infty) , \\
u=0 & \text{on} &\partial \Omega\times (0,+\infty),\\
u(x,0)=u_{0}(x) & \text{in} &\Omega , 
\end{array}\right.\tag{P}
\end{equation}
Under some assumptions on the initial data $u_{0}$, they used the nonlinear semigroup theory to prove the existence and uniqueness of solutions as well as they also described the behavior of solutions to the above problem near the extinction time.

The objective of this paper is to discuss the existence and asymptotic behavior of solutions near the extinction time to the following class of Kirchhoff type problem involving the $1$-Laplace operator 

\begin{equation}\label{E1}
\left\{\begin{array}{llc}
u_{t}-m\left(\int_{\Omega}|Du|\right)\Delta_{1} u=0 & \text{in}\ & \Omega\times (0,+\infty) , \\
u=0 & \text{on} &\partial \Omega\times (0,+\infty),\\
u(x,0)=u_{0}(x) & \text{in} &\Omega , 
\end{array}\right.
\end{equation}
where  $\Omega \subset \mathbb{R}^N$ is a smooth bounded domain, $N \geq 1$  and $\Delta_1 u=\text{div}\left(\frac{Du}{|Du|}\right)$ denotes the 1-Laplace operator. In the literature,  equations in (\ref{P})-(\ref{E1}) are  also known  as very singular diffusion equations, see, for example  Giga, Giga and Kobayashi \cite{Giga1} and their references. 

These kinds of problems have attracted much attention in recent years due to their applications in image processing, faceted crystal growth, continuum mechanics, for details see  \cite{And} and  \cite{Giga1}. In geometry, as it is well known the unit normal of the level set $\{u=k\}$ is given formally by $\eta(x)=\frac{Du}{|Du|}$, then the mean curvature of this surface  is formally given by 
$$ H(u)=\text{div}(\eta)(x)=\text{div}\left(\frac{Du}{|Du|}\right),$$
which turns out that  the solution of a parabolic 1-laplacian problem can be also seen as a solution to the evolution mean curvature flow for the level sets $\{u=k\}$, see Ecker \cite{KE} for more details. 

 As we observe from (\ref{E1}), the differential operator in the left-hand side of the equation in (\ref{P}) has a singularity at $|Du|=0$, which introduces some extra difficulties and special features. Another difficulty here is to give sense to the boundary condition $u=0$ on $\partial \Omega$ which is in general does not necessarily hold in the sense of trace. 
 
 In order to handle such difficulties, Andreu, Ballester, Caselles, and Maz\'on \cite{Andreu, Andreu3} proposed a definition of  weak solution  based on the
 so-called Anzellotti
 pairing $(z, Du)$ of $L^{\infty}$-divergence-measure vector field $z$ and gradient of a $BV$ function $u$. In their definition the bounded vector field $z$ considered as a substitute of $\frac{Du}{|Du|}$, while the boundary condition $u=0$ on $\partial \Omega$ is taken in a very weak sense (see Definition \ref{Def1}). After the concept of solutions becomes known in the literature, there are two approaches used frequently to show the existence  of solutions to total variation flow  problems. The first one is based on the nonlinear semigroup theory, in particular on techniques of completely accretive operators and the Crandall-Liggett semigroup generation theorem, and when this approach has been used,  different assumptions on the initial data $u_{0}$ assumed to establish the existence of solutions, see for example  \cite{Andreu, Andreu3, And, DH1, DH2, Giga1}. In \cite{ANDR}, Andreu,  Maz\'on, and Mollthe extended the work presented in \cite{Andreu} for (\ref{P}) to the following total variation follow with nonlinear boundary conditions 
 
 \begin{equation}\label{P2}
 \left\{\begin{array}{llc}
 u_{t}-\text{div}\left(\frac{Du}{|Du|}\right)=0 & \text{in}\ & \Omega\times (0,+\infty) , \\
 -\frac{\partial u}{\partial \nu}\in \beta(u) & \text{on} &\partial \Omega\times (0,+\infty),\\
 u(x,0)=u_{0}(x) & \text{in} &\Omega , 
 \end{array}\right.\tag{NP}
 \end{equation}
 where $\Omega$ is an open bounded domain in $\mathbb{R}^{N}$ with a $C^{1}$ boundary, $\partial /\partial \nu$ is the Neumann boundary operator associated to $\frac{Du}{|Du|}$, i.e., 
 $$ \frac{\partial u}{\partial \nu} :=\left\langle \frac{Du}{|Du|}, \nu\right\rangle,$$
 with $\nu$ is the unit outward normal vector on $\partial \Omega$ and $\beta$ is a maximal monotone graph in $\mathbb{R}\times \mathbb{R} $ with $0\in \beta(0)$. In this work by using the nonlinear semigroup theory and the Anzellotti pairing the authors proved the following interesting results : 
 \begin{itemize}
 	\item If $u_{0}\in L^{2}(\Omega)$, then there exists a unique strong solution $u(t)$ of  (\ref{P2}), 
 	\item if $u_{0}\in L^{1}(\Omega)$, then there exits a unique entropy solution of (\ref{P2}). 
 \end{itemize} 
Furthermore, they provided some explicit solutions to (\ref{P2}). Maz\'on, Rossi and Segura de Le\'on \cite{DH3}
 considered the following problem with dynamical boundary condition 
 \begin{equation}\label{PD}
 \left\{\begin{array}{llc}
 -\text{div}\left(\frac{Du}{|Du|}\right)=0 & \text{in}\ & \Omega\times (0,+\infty) , \\
 u_{t}+\left[\frac{Du}{|Du|},\nu\right]=g, & \text{on} &\partial \Omega\times (0,+\infty),\\
 u(x,0)=u_{0}(x) & \text{in} &\Omega, 
 \end{array}\right.\tag{PD}
 \end{equation}
 and proved the existence and uniqueness of a semigroup solution of (\ref{PD}) provided that $u_{0}\in L^{2}(\Omega)$ and $g\in L^{2}(0, T; L^{2}(\Omega))$. 
 
 The second strategy is based on taking the limit as $p\rightarrow 1^{+} $ of solutions to the following problem 
 
 \begin{equation}\label{Pp}
 \left\{\begin{array}{llc}
 u_{t}-\text{div}\left(|\nabla u|^{p-2}\nabla u\right)=0 & \text{in}\ & \Omega\times (0,+\infty) , \\
 u=0 & \text{on} &\partial \Omega\times (0,+\infty),\\
 u(x,0)=u_{0}(x) & \text{in} &\Omega , 
 \end{array}\right.\tag{$P_{p}$}
 \end{equation}
As far as we know, this approach has been used for the first time by  Leon and Webler in \cite{SW} where the authors
 studied global existence and uniqueness of solutions for an  inhomogeneous problem related to (\ref{P14}). Namely, they considered the following problem 
\begin{equation}\label{P14}\left\{
\begin{array}{llc}
u_{t}-\text{div}\left(\frac{Du}{|Du|}\right)= f(x,t)& \text{in}\ & \Omega\times (0, +\infty) , \\
u =0 & \text{in} & \partial\Omega\times (0, +\infty), \\
u(x,0)=u_{0}(x)& \text{in} &\Omega , 
\end{array}\right.
\end{equation}
and proved global existence and uniqueness of solutions for (\ref{P14}) via a parabolic $p$-laplacian problem and then taking the limit $p\rightarrow 1^{+}$
where $u_{0} \in L^{2}(\Omega)$ and $f\in L_{loc}^{1}(0, +\infty;L^{2}(\Omega) )$.  Gianazza and Klaus \cite{GZZ1}  showed  that the solution of (\ref{Pp}) converges to a solution of the total variation flow problem (\ref{P}). Subsequently in \cite{Alves}, Alves and Boudjeriou  considered the following problem 
\begin{equation}\label{AT}\left\{
\begin{array}{llc}
u_{t}-\text{div}\left(\frac{Du}{|Du|}\right)= f(u)& \text{in}\ & \Omega\times (0, +\infty) , \\
u =0 & \text{in} & \partial\Omega\times (0, +\infty), \\
u(x,0)=u_{0}(x)& \text{in} &\Omega , 
\end{array}\right.
\end{equation}
where $f(u)$ behaving like these functions $f(u)=|u|^{q-2}ue^{\alpha |u|^2}$ for $q>1$, $\alpha>0$ and   $f(u)=|u|^{q-2}u+|u|^{s-2}u$ with $q,s \in (1,N/(N-1))$, and established the existence of global solutions by taking the limit as $p\rightarrow 1^{+}$ of solutions to a parabolic $p$-Laplace problem related to (\ref{AT}). It is worth mentioning that the approximation of proper solutions to the $1$-laplacian in terms of solutions to the $p$-laplacian as $p\rightarrow 1^{+}$ has been also used to treat this stationary problem 
\begin{equation*}\left\{
\begin{array}{llc}
-\text{div}\left(\frac{Du}{|Du|}\right)= f(u)& \text{in}\ & \Omega, \\
u =0 & \text{in} & \partial\Omega, 
\end{array}\right.
\end{equation*}
where $\Omega \subseteq \mathbb{R}^{N}$ and $f :\mathbb{R} \rightarrow \mathbb{R}$ is a continuous function satisfies some conditions. For example, we refer the reader to Alves \cite{Alves0}, Juutinen \cite{Juu}, Mercaldo et al. \cite{MRST} and \cite{MST}, Molino Salas and Segura de Le\'on \cite{sergio}, Kawohl and  Schuricht \cite{Kawohl} and the references therein. 

In conclusion, we point out that B\"ogelein, Duzaar, and Marcellini in \cite{VBD1,VBD2} developed a new approach based on a parabolic variational inequality to solve some classes of total variation flow. In this mentioned approach, the Anzellotti-pairing plays no role and the theory of nonlinear semigroup is ignored. Very recently, Kinnunen and Scheven in \cite{JKN} showed that weak solutions to the total variation flow based on the Anzellotti pairing and the approach due to B\"ogelein, Duzaar and Marcellini are in fact equivalent under natural assumptions. 

When one compares problem (\ref{E1})  with the previous ones there are important differences, for example, the differential operator in (\ref{E1})  is not homogeneous of zero degree and so many techniques used in \cite{Andreu2} are not applicable here. It is important to stress that this is the first time the Kirchhoff total variation flow is studied. 

Throughout the paper, we shall assume that $m :\mathbb{R}_{+}\rightarrow \mathbb{R}_{+}$  is an increasing continuous function that satisfies  :
$$
 m(r)\geq m(0)>0\;\text{for every}\; r\geq 0. \eqno(m_1)
$$ 
In what follows, we denote by $ L_{w}^{1}(0, T; BV(\Omega))$ the space of functions $u : [0, T]\rightarrow BV(\Omega)$ such that $u \in L^{1}(\Omega \times (0,T))$, the maps $t\mapsto \langle Du(t), \varphi\rangle$ is measurable for every $\varphi \in  C_{0}^{1}(\Omega, \mathbb{R}^{N})$ and such that $ \int_{0}^{T}\int_{ \Omega} |Du(t)|\,dt< \infty$. The space $BV(\Omega)$ will be introduced in the next section. We denote the positive part of $w$ by $w^{+}=\max\{w,0\}$. These notations $ u_{t}(t)$ or $u'(t)$ will be used to denote the derivative of $u$ with respect to time $t$.

Now, for the reader's convenience, we give the definition of solutions to (\ref{E1}) based on the Anzellotti pairing.
\begin{definition}\label{Def1} Let $u_{0} \in L^{2}(\Omega)\cap \text{BV}(\Omega)$. A function $u\in C([0,T], L^{2}(\Omega))$ will be called a strong solution of (\ref{E1}) if $u_{t}\in L^{2}(0, T;L^{2}(\Omega))$, $u\in L_{w}^{1}(0, T;\text{BV}(\Omega)) $ and there exist $z(t)\in X(\Omega)$, $\|z(t)\|_{\infty}\leq 1$, satisfying
	
	\begin{enumerate}
		\item $u'(t)-m\left(\int_{\Omega}|Du(t)|\right)\text{div}(z(t))=0, \;\;\text{in}\;\mathcal{D}'(\Omega) \;\;\text{a.e}\;t\in [0,T],$
		\item $ \int_{\Omega} (z(t), Du(t))	=\int_{\Omega}|Du(t)|,$
		\item $[z(t), \nu]\in \text{sign}(-u(t))\;\;\mathcal{H}^{N-1}-\text{a.e}\;\;\text{on}\;\partial \Omega.$
	\end{enumerate}
\end{definition}
The first result of this paper reads as follows : 
\begin{theorem}\label{Th1}
Let $u_{0}\in BV(\Omega)\cap L^{\infty}(\Omega)$ and assume that $(m_{1})$ holds. Then Problem (\ref{E1}) has a unique solution $u$ in the sense of Definition \ref{Def1}, which satisfies 
\begin{equation}\label{EN}
	M\left( \int_{ \Omega} |Du(t)|+\int_{\partial\Omega} |u(t)|\,d\mathcal{H}^{N-1}\right)\leq M\left( \int_{ \Omega} |Du_{0}|+\int_{\partial\Omega} |u_{0}|\,d\mathcal{H}^{N-1}\right) \;\;\; \forall t\geq 0, 
\end{equation}
where $M(\sigma)=\int_{0}^{\sigma}m(s)\,ds$. Moreover, assuming that there is $r>0$ such that $B_{r}(0)\subset\subset \Omega$. Then the following conclusion holds :
\begin{itemize}
	\item If $u_{0}(x)=k\chi_{B_{r}(0)}(x)$ and $m(\sigma)=(\sigma+1)^{p}$ for $\sigma\geq 0$ with $p> 1$, $k>0$. Then the explicit solution of (\ref{E1}) is given by 
	\begin{equation}\label{Sl1}
		u(x,t)=\frac{\left(\left[N(p-1)\gamma_{N}r^{N-2}t+(\gamma_{N} kr^{N-1}+1)^{1-p}\right]^{\frac{1}{1-p}}-1\right)^{+}}{\gamma_{N}r^{N-1}}\chi_{B_{r}(0)}(x)\chi_{[0,T]}(t).
	\end{equation}
where $T=\frac{1-(\gamma_{N} kr^{N-1}+1)^{1-p}}{N(p-1)\gamma_{N}r^{N-2}}$, $\gamma_{N}=\frac{N\pi^{N/2}}{\Gamma \left(\frac{N}{2}+1\right)}$. 

	\item If $u_{0}(x)=k\chi_{B_{r}(0)}(x)$ and $m(\sigma)=1+\sigma$ for $\sigma\geq 0$, $k>0$. Then the explicit solution of (\ref{E1}) is given by
	\begin{equation}\label{Sl2}
	u(x,t)=\frac{\left(e^{-N\gamma_{N}r^{N-2}t}(\gamma_{N}r^{N-1}k+1)-1\right)^{+}}{\gamma_{N}r^{N-1}}\chi_{B_{r}(0)}(x)\chi_{[0,T]}(t),	
	\end{equation}
where $T=\frac{\log(\gamma_{N}r^{N-1}k+1)}{N\gamma_{N}r^{N-2}}.$ 
\end{itemize}
\end{theorem}
\begin{remark}\label{ReM}
	From \cite [Lemma 1]{Andreu}, we observe that solving Problem (\ref{E1}) in the sense of Definition \ref{Def1} is equivalent to solve the following Cauchy problem 
	\begin{equation}\label{E11}
	\left\{\begin{array}{llc}
	u_{t}(t)+m\left(\int_{\Omega} |Du(t)|\right)A(u(t))\ni  0 & \text{in}\ & t>0 , \\
	u(0)=u_{0}, & & 
	\end{array}\right.
	\end{equation}
	where $A(u)=\partial \Phi(u)$ and 
	$$
	\Phi(u)=\left\{
	\begin{array}{l}
	\displaystyle \int_{\Omega} |Du|+\int_{\partial\Omega} |u|\, d\mathcal{H}^{N-1},\quad \mbox{if} \quad u \in \text{BV}(\Omega) \cap H,\\
	\mbox{}\\
	+\infty, \quad \mbox{if} \quad  u \in H \setminus \text{BV}(\Omega) \cap H.
	\end{array}
	\right.
	$$ 
	Here and throughout the paper $H=L^{2}(\Omega)$.
\end{remark}
\begin{remark}
From (\ref{Sl1})-(\ref{Sl2}), we observe that the solution of (\ref{E1})  is discontinuous and has the minimal required spatial regularity $u(x,t)\in BV(\Omega)\backslash W^{1,1}(\Omega).$ We point out that the regularity of solutions to problem (\ref{P})  still an open question. 
\end{remark}
\begin{remark}
By using the method of speratinng variables, the authors in  (\cite{Andreu2}, \cite{Kawohl}) have observed that every solution for a parabolic 1-laplacian problem  must decays to zero in finite time $T$. In view of (\ref{Sl1})-(\ref{Sl2}), we observe that the same conclusion still holds true for the solutions of (\ref{E1}). 
\end{remark}
In order to continue with this new line of research, we suggest some open questions, which we consider to be interesting :
\begin{itemize}
\item We are wondering if the result stated in Theorem \ref{Th1} still holds in the case when $m(0) =0$. 
\item 	Let $u_{0}, u_{1} :\Omega \rightarrow \mathbb{R}$ be smooth functions. Can we prove the existence of  a function $u$ in the sense of Definition \ref{Def1} obeying the equations
\begin{equation}
\begin{array}{llc}
u_{tt}-\text{div}\left(\frac{Du}{|Du|}\right)=0 & \text{in}\ & \Omega\times (0,+\infty) , \\
u=0 & \text{on} &\partial \Omega\times (0,+\infty),\\
u(x,0)=u_{0}(x), \; u_{t}(x,0)=u_{1}(x)& \text{in} &\Omega?.
\end{array}
\end{equation}
\end{itemize}
The proof of the first part of Theorem \ref{Th1} will be based on the existence of solutions for (\ref{P})  and the existence of solutions for an ordinary differential equation. This argument has been introduced by Chipot and Lovat \cite{ML} to prove the existence and uniqueness of solution for this nonlocal problem involving the Laplace operator 
\begin{equation*}
\left\{\begin{array}{llc}
u_{t}-a\left(\int_{\Omega} u(x,t)\,dx\right)\Delta u=0 & \text{in}\ & \Omega\times (0,+\infty) , \\
u=0 & \text{on} &\partial \Omega\times (0,+\infty),\\
u(x,0)=u_{0}(x) & \text{in} &\Omega , 
\end{array}\right.
\end{equation*}
where $a$ is assumed to be continuous and $a(s)>0$ for all $s>0$. Later on, this approach have been used by Gobbino \cite{MG4} to  extend Chipot and Lovat result to an abstract setting. Namely, he discussed the existence and uniqueness of solution for this initial value problem 
 \begin{equation}\label{ESS}
 \left\{\begin{array}{llc}
 u_{t}+m\left( \|A^{1/2}u\|^{2}\right)A u=0 & \text{in}\ & t>0 , \\
u(0)=u_{0}\in H, & & 
 \end{array}\right.
 \end{equation}
 where $A$ is a self-adjoint linear non-negative operator on $H$ with domain $D(A)$. The function $m$ is continuous and behaving like $m(\sigma)=a+b\sigma$, $a\geq0 $, $b>0$. In this work the author proved the existence of at least one solution for (\ref{ESS}) when $m\left( \|A^{1/2}u_{0}\|^{2}\right)\neq 0$. The construction of explicit solutions to (\ref{E1}) is inspired by a work due to Andreu, Caselles, D\'iaz and Maz\'on \cite{Andreu2} where the authors have been  studied the parabolic $1$-laplacian problem where $m\equiv 1.$
 
Before introducing the second result, we need the following definition : 
\begin{definition}
We say that Problem (\ref{E1}) has the extinction in finite time property if there exists a number $T_{0}> 0$ such that $u(x,t)=0$ when $t\geq T_{0}$ and for any $x\in \Omega$.	
\end{definition}

\begin{theorem}\label{Th2}
	Let $u_{0}\in BV(\Omega)\cap L^{\infty}(\Omega)$ and let $u(x,t)$ be a unique solution of Problem (\ref{E1}). Assume that $(m_{1})$ holds and let $d(\Omega)$ be the smallest radius of  a ball containing $\Omega$. If $T^{*}(u_{0})=\inf\{t> 0 :\, u(t)=0\}$, then 
	\begin{equation}
	T^{*}(u_{0})\leq \frac{d(\Omega)\|u_{0}\|_{\infty}}{Nm(0)}.
	\end{equation}
	Moreover, there holds 
	\begin{equation}\label{BO}
		\|u(t)\|_{\infty}\leq \|u_{0}\|_{\infty}, \;\;\text{for any} \;t\geq 0. 	
	\end{equation}
\end{theorem}
Extinction of solutions  for evolutionary equations has been widely studied in literature. As far as we know, the first proof of finite-time extinction for the total variation flow (\ref{P}) was given by Andreu, Caselles, Diaz,
and Maz\'on in \cite{Andreu2}. In that paper, the authors proved a comparison principle for total variation flow (\ref{P}), and obtained that
\begin{equation*}
T^{*}(u_{0})\leq \frac{d(\Omega)\|u_{0}\|_{\infty}}{N}.
\end{equation*}
Our approach to prove Theorem \ref{Th2} is also based on a comparison principle established in this paper. Recently, Giga and Kohn in \cite{Giga1} rather than using a comparison principle  they used an energy estimate with a suitable Sobolev-type inequality to prove the finite-time extinction of solutions to (\ref{P}) under different boundary conditions (Periodic BC, Neumann BC, Dirichlet BC). Namely, they showed that
for every $u_{0}\in L_{av}^{2}(\mathbf{T}^{N})$, the extinction time satisfies
$$ T^{*}(u_{0})\leq S_{N}\|u_{0}\|_{L^{N}}, $$ 
where $L_{av}^{2}(\mathbf{T}^{N})=\{v\in L^{2}(\mathbb{T}^{N}) :\, \int_{\mathbf{T}^{N}}v\,dx=0 \}$ and $S_{N} $ is the best constant in the Sobolev inequality. Bonforte and Figalli \cite{BF} studied the explicit dynamic
and sharp asymptotic behaviour for the following Cauchy problem of total variation flow in one dimension 
 \begin{equation*}
\left\{\begin{array}{llc}
u_{t}-\left(\frac{u'}{|u'|}\right)'=0 & \text{in}\ & \mathbb{R}\times (0,+\infty) , \\
u(x,0)=u_{0}(x) & \text{in} &\mathbb{R}, 
\end{array}\right.
\end{equation*}
Moreover, for any given nonnegative compactly supported initial datum  $ u_{0}\in L^{1}(\mathbb{R})$, the authors proved that 
$$ T^{*}(u_{0})=\frac{1}{2}\int_{\mathbb{R}}u_{0}(x)\,dx.$$
 The first part of the next theorem demonstrates some lower and  upper bounds for the solution of the Kirchhoff total variation flow with different choices of the function $m$. While the second part shows that  there is no propagation
 of the support of the initial datum. We point out that it has been proved that if $p> 2$, then there is finite speed propagation property of solutions to (\ref{Pp}), that is,  if $supp (u_{0})\subset B_{r}(0)\subset\subset \Omega$, then the solution of (\ref{Pp}) satisfies that $supp (u(t)) $ is a compact set for all $t> 0$; for more details see, \cite{JDH} and \cite{VAZ}.
 
 \begin{theorem}\label{TH3}
 	Let $u_{0}\in BV(\Omega)\cap L^{\infty}(\Omega)$ and denote by $u(t)$ the solution of (\ref{E1})  with initial data $u_{0}$. Assume  that there is $r>0$ such that $B_{r}(0)\subset\subset \Omega$. Then the following conclusion hold :
 	\begin{itemize}
 		\item If $m(\sigma)=(\sigma+1)^{p}$ for $ \sigma \geq 0$ and $p> 1$, then 
 		\begin{equation}\label{IN1}
 		\|u(t)\|_{\infty} \geq\frac{\left(\left[N(p-1)\gamma_{N}r^{N-2}t+(\gamma_{N} kr^{N-1}+1)^{1-p}\right]^{\frac{1}{1-p}}-1\right)^{+}}{\gamma_{N}r^{N-1}}\;\;\text{for}\;\; 0\leq t\leq T_{1}^{*}(u_{0}), 
 		\end{equation} 
 		where $T_{1}^{*}(u_{0})=\frac{1-(\gamma_{N} kr^{N-1}+1)^{1-p}}{N(p-1)\gamma_{N}r^{N-2}}$, $\gamma_{N}=\frac{N\pi^{N/2}}{\Gamma \left(\frac{N}{2}+1\right)}$. Moreover, if $supp (u_{0})\subset B_{r}(0)\subset\subset \Omega$, then $supp (u(t))\subset B_{r}(0) $ for all $t\geq 0$, and $u(t)=0$ for all $t\geq \frac{1-(\gamma_{N} \|u_{0}\|_{\infty}r^{N-1}+1)^{1-p}}{N(p-1)\gamma_{N}r^{N-2}}$.
 		\item If $m(\sigma)=\sigma+1$ for $ \sigma \geq 0$, then 
 		\begin{equation}\label{IN2}
 		\|u(t)\|_{\infty} \geq \frac{\left(e^{-N\gamma_{N}r^{N-2}t}(\gamma_{N}r^{N-1}k+1)-1\right)^{+}}{\gamma_{N}r^{N-1}} \;\;\text{for}\;\; 0\leq t\leq T_{2}^{*}(u_{0}), 	
 		\end{equation}
 		where $T_{2}^{*}(u_{0})=\frac{\log(\gamma_{N}r^{N-1}k+1)}{N\gamma_{N}r^{N-2}}.$ Furthermore, if $supp (u_{0})\subset B_{r}(0)\subset\subset \Omega$, then $supp (u(t))\subset B_{r}(0) $ for all $t\geq 0$, and $u(t)=0$ for all $t\geq \frac{\log(\gamma_{N}r^{N-1}\|u_{0}\|_{\infty}+1)}{N\gamma_{N}r^{N-2}}.$
 	\end{itemize}
  \end{theorem}
In the last theorem of this paper,  we establish  lower and upper bounds on  the rate of decay of $\|u(t)\|_{N}$ and $\|u(t)\|_{\infty}$ respectively. 

\begin{theorem}\label{Th3}
	Let $u_{0}\in BV(\Omega) \cap L^{\infty}(\Omega)$ and $u(x,t)$ is the unique solution of problem (\ref{E1}). Assume that there exists $\mu \in (0,1)$ such that
	\begin{equation}\label{CN}
	M(\sigma)	\geq \mu m(\sigma)\sigma \;\;\text{for every }\; \sigma\geq 0.
	\end{equation}
	Then we have 
	\begin{itemize}
		\item There exists a constant $\eta>0$ independent of the initial datum, such that 
		\begin{equation}\label{INQ}
		\|u(t)\|_{N}\geq \eta \min\{m_{0}, 1\}N(T^{*}(u_{0})-t)\;\;\text{for}\;\; 0\leq t\leq T^{*}(u_{0}).	
		\end{equation}
		\item Given $0< \tau < T^{*}(u_{0})$, we have 
		\begin{equation}
		\|u(t)\|_{\infty}\leq \frac{|u_{0}|m\left( \frac{1}{\mu m_{0}}M\left(\int_{\Omega} |Du_{0}|+\int_{\partial \Omega} |u_{0}|\,d\mathcal{H}^{N-1}\right)\right)}{\tau}(T^{*}(u_{0})-t) \;\;\text{for}\;\; \tau \leq t\leq T^{*}(u_{0}).
		\end{equation}
	\end{itemize}
\end{theorem} 

The paper is organized as follows. In Section $2$, we give the necessary definitions and some preliminary results on functions of bounded variation. In the other sections, we prove our results.
\section*{Acknowledgments} The author would like to thank Professor Claudianor Alves for several interesting and useful discussions on this topic.

\section{Notation and preliminaries involving the space $BV(\Omega)$}
Before giving proofs to the main results of this paper, we  begin by introducing some notations and recalling several facts on functions of bounded variation.

Throughout the paper, without further mentioning, given an open bounded set $\Omega$ in $\mathbb{R}^{N}$ with Lipschtiz boundary, we denote by $\mathcal{H}^{N-1} $ the $(N-1)-$dimensional Hausdorff measure  and $|\Omega|$ stands for the $N$-dimensional Lebesgue measure. Moreover, we shall denote by $\mathcal{D}(\Omega)$ or $C^{\infty}_{0}(\Omega)$, the space of infinitely differentiable functions with compact support in $\Omega$ and $\nu (x) $ is the outer vector normal defined for $\mathcal{H}^{N-1}$- almost everywhere $x\in \partial \Omega$. 

We will denote by $BV(\Omega)$ the space of functions of bounded variation 
$$
BV(\Omega)=\left\{u\in L^{1}(\Omega):\, Du \;\text{is a bounded Radon measure}\right\},
$$
where $Du :\Omega \rightarrow \mathbb{R}^{N}$ denotes the distributional gradient of $u$.  It can be proved that $ u\in {BV}(\Omega)$ is equivalent to $u\in L^{1}(\Omega)$ and 
$$ \int_{\Omega}|Du|:=\text{sup}\left\{\int_{\Omega} u\,\text{div}\varphi\,dx :\, \varphi \in C^{\infty}_{0}(\Omega, \mathbb{R}^{N}),\, |\varphi (x)|\leq 1\; \forall x\in \Omega\right\}< +\infty,$$
where $|Du|$ is the total variation of the vectorial Radon measure. We recall that the space ${BV}(\Omega)$ endowed with the norm 
$$
\|u\|_{{BV}(\Omega)}:=\int_{\Omega} |Du|+\|u\|_{L^{1}(\Omega)},
$$ 
is a Banach space which is  non-reflexive and non-separable. For more information on functions of bounded variation we refer the reader to \cite{Ambrosio}, \cite{Evn} and \cite{Ziem}. 

In view of \cite[Theorem 3.87]{Ambrosio}, the notion of a trace on the boundary can be extended to functions $u\in {BV}(\Omega)$, through a bounded operator ${BV}(\Omega)\hookrightarrow L^{1}(\partial \Omega)$, which is also onto. As a consequence, an equivalent norm on ${BV}(\Omega) $ can be defined by
$$ 
\|u\|:=\int_{\Omega} |Du|+\int_{\partial\Omega} |u|\, d\mathcal{H}^{N-1}.
$$ 
In addition, by \cite[Corollary 3.49]{Abm} the following continuous embeddings hold 
\begin{equation}\label{EMD}
{BV} (\Omega)\hookrightarrow L^{m}(\Omega)\;\;\text{for every}\;\;1\leq m\leq1^{*}=\frac{N}{N-1},
\end{equation}  
which are compact for $1\leq m<1^{*}$.  

In what follows, we recall several important results from \cite{Anz}  which will be used throughout the paper. Following \cite{Anz}, let 
\begin{equation}
X(\Omega)=\left\{z\in L^{\infty}(\Omega, \mathbb{R}^{N}):\, \text{div}(z)\in L^{1}(\Omega) \right\}.
\end{equation}

If $z\in  X(\Omega)$ and $w\in {BV}(\Omega)$ we define the functional $(z, Dw): C^{\infty}_{0}(\Omega)\rightarrow \mathbb{R}$ by formula 
\begin{equation}
\langle (z, Dw),\varphi\rangle=-\int_{\Omega} w\varphi \text{div} (z)\,dx -\int_{\Omega}wz.\nabla \varphi\,dx, \;\;\forall \varphi \in C^{\infty}_{0}(\Omega).
\end{equation}
Then, by \cite[Theorem 1.5]{Anz} $(z, Dw)$ is a Radon measure in $\Omega$, 
$$ \int_{ \Omega} (z, Dw)=\int_{ \Omega} z.\nabla w\,dx, $$
for all $w\in W^{1,1}(\Omega)$  and 
\begin{equation}\label{BR}
\left|	\int_{\Omega} (z, Dw)\right|\leq\int_{B}|(z, Dw)|\leq \|z\|_{\infty}\int_{B} |Dw|.
\end{equation}
for every Borel $B$ set with $B\subseteq  \Omega $. Furthermore, the measures $(z, Dw)$ and $|(z, Dw)|$ are absolutely continuous with respect to the measure $|Dw|$. Denoting by 
$$ \theta (z, Dw, .) :\Omega\rightarrow \mathbb{R}$$ 
the Radon-Nikod\'ym derivative of $(z, Dw)$ with respect to $|Dw|$, it follows that 
\begin{equation}
\int_{B} (z, Dw)=\int_{B} \theta (z, Dw, x)|Dw|, 
\end{equation}
for any Borel set $B\subseteq \Omega$ and 
\begin{equation}
\|\theta(z, Dw, .)\|_{L^{\infty}(\Omega, |Dw|)}\leq \|z\|_{\infty}.
\end{equation}
On the other hand, besides the $BV-$norm, for any nonnegative smooth function $\varphi $  the functional given by 
$$ w \mapsto \int_{\Omega}\varphi |Dw|, $$ 
is lower semicontinuous with respect to the $L^{1}-$convergence, for details see \cite{Ambrosio}.

In \cite{Anz}, a weak trace on $\partial \Omega$  of normal component of $z\in X(\Omega) $ is defined as the application $[z,\nu] : \partial \Omega \rightarrow \mathbb{R}$, such that $[z,\nu]\in L^{\infty}(\partial \Omega)$ and $\|[z,\nu]\|_{\infty}\leq \|z\|_{\infty}$. In addition, this definition coincides with the classical one, that is 
\begin{equation}
[z,\nu]=z.\nu, \;\text{for}\;z\in C^{1}(\overline{\Omega_{\delta}}, \mathbb{R}^{N}), 
\end{equation} 
where $\Omega_{\delta}=\{x\in \Omega :\, d(x, \Omega)< \delta\}$, for some $\delta > 0$ sufficiently small. We recall the Green formula involving the measure $(z,Dw)$ and the weak trace $[z,\nu]$ which was given in \cite{Anz}, namely :
\begin{equation}\label{Gree}
\int_{\Omega} (z, Dw)+\int_{\Omega} w\text{div} z\,dx=\int_{\partial\Omega} w[z,\nu]\, d\mathcal{H}^{N-1},
\end{equation}
for $z\in X(\Omega)$ and $w\in {BV} (\Omega).$

\section{ existence of solutions to (\ref{E1}) }
This section is devoted  to the proof of Theorem \ref{Th1}. The idea of the proof is based on the following observation :   
\begin{equation}\label{EW}
	u \; \text{is a solution of }\; (\ref{E1})\; \text{ if, and only if, }\,u(t)=v(\alpha(t)),
\end{equation} 
where $v$ is a solution for the following  total variation flow problem  
\begin{equation}\label{E2}
\left\{\begin{array}{llc}
v_{t}-\text{div}\left(\frac{Dv}{|Dv|}\right)=0 & \text{in}\ & \Omega\times (0,+\infty) , \\
v=0 & \text{on} &\partial \Omega\times (0,+\infty),\\
v(x,0)=u_{0}(x) & \text{in} &\Omega , 
\end{array}\right.
\end{equation}
and $\alpha$ is a solution of the problem 
\begin{equation}\label{ES}
\left\{\begin{array}{l}
\alpha \in C([0, +\infty)) \cap C^{1}(]0, +\infty)), \\
\alpha'(t)=\varphi(\alpha(t))\;\;\;t> 0, \\
\alpha(t)> 0, \;\;\;t> 0, \\
\alpha(0)=0,
\end{array}\right.
\end{equation} 
where $\varphi(t)=m\left(\int_{\Omega} |Dv(t)|\right)$.

It has proven in \cite{Andreu} that solving problem (\ref{E2}) is equivalent to solving the following Cauchy problem : 
\begin{equation}\label{E3}
\left\{\begin{array}{llc}
v'(t)+A(v(t))\ni  0 & \text{in}\ & 0< t< +\infty , \\
v(0)=u_{0}, & & 
\end{array}\right.
\end{equation}
where $A(v)=\partial \Phi(v)$ and $\Phi(v)$ has been introduced in Remark \ref{ReM}. In view of  $(m_{1})$ and the theory of ordinary differential equations problem (\ref{ES}) has a unique solution given by 
$$ \alpha(t)=\psi^{-1}(t)\;\;\text{where}\; \psi(r)=\int_{0}^{r}\frac{1}{\varphi(s)}\,ds.,$$
 for details see Gobbino \cite{MG4}.  Now let us show (\ref{EW}). Indeed, let $\alpha$ be a solution of (\ref{ES}) and $v$ be a unique solution of (\ref{E3}), let $u(t)=v(\alpha(t))$. Then $u(0)=v(\alpha(0))=v(0)=u_{0}$, and for every $t>0$ we have 
$$ u'(t)=\alpha'(t)v'(t)\in - \alpha'(t)Av(\alpha(t))=-m\left(\int_{\Omega} |D u(t)|\right)A(u(t)).$$ 
This proves that $u$ is a solution of (\ref{E11}) and the regularity of $u$ follows from the regularity of $v$ and $\alpha$. Conversely,  let $u$ be a solution of (\ref{E11}) and let us consider the continuous function $b(t)=m\left(\int_{\Omega} |D u(t)|\right)$ defined for all $t>0$. From $(m_{1})$, we have that $b(t)$ is not identically zero in a right neighbourhood of $t=0$. Moreover, $u$ is a solution for the  initial value problem  

\begin{equation*}
\left\{\begin{array}{llc}
u'(t)+b(t)A(u(t))\ni  0 & \text{in}\ & t>0 , \\
u(0)=u_{0}, & & 
\end{array}\right.
\end{equation*}
This implies that $u(t)=v(\alpha(t))$ where 
$$ \alpha(t)=\int_{0}^{t} b(s)\,ds.$$
Now, it is easy to show that $\alpha$  is a solution to (\ref{ES}).  Finally, by using \cite[Lemma p.73]{H} we conclude that (\ref{EN}) holds.

In the last part of this section we compute the explicit solution to Problem (\ref{E1}). In doing so, we look for a solution of (\ref{E1}) of the form $u(x,t)=\alpha(t)\chi_{B_{r}(0)}(x)$ on some interval $(0, T)$. Then we shall look for some $z(t)\in X(\Omega)$ with $\|z(t)\|_{\infty}\leq 1$, such that 
\begin{equation}\label{Eq1}
u'(t)-m\left(\int_{\Omega} |Du(t)|\right)\text{div}(z(t))=0\;\;\text{in}\; \mathcal{D}'(\Omega), 
\end{equation}
\begin{equation}\label{Eq2}
\int_{ \Omega}(z(t), Du(t))=\int_{\Omega} |Du(t)|, 
\end{equation}
\begin{equation}\label{Eq3}
[z(t), \nu]\in \text{sign}(-u(t))\;\;\;\mathcal{H}^{N-1}-a.e.
\end{equation}
The candidate to $z(t) $ is the vector field 
\begin{equation}\label{VC}
z(t)(x):=\left\{\begin{array}{lcc}
\frac{-x}{r} &\text{if} & x\in B_{r}(0), \;  0\leq t\leq T , \\
\frac{-r^{N-1}x}{|x|^{N}} &\text{if} & x\in \Omega \backslash B_{r}(0), \; 0\leq t\leq T, \\
0 &\text{if} & x\in \Omega, \, t>T, 
\end{array}\right.
\end{equation}
The computations to construct (\ref{VC}) was carried out in \cite[Lemma 1]{Andreu2}, so here we omit it. 
 Integrating Eq.(\ref{Eq1}) over $\Omega$  yields
$$ \alpha'(t)|B_{r}(0)|=-\left[\alpha(t)\gamma_{N}r^{N-1}+1\right]^{p}N\gamma_{N}r^{N-1},$$
where $\gamma_{N}=\frac{N\pi^{N/2}}{\Gamma \left(\frac{N}{2}+1\right)}$. Therefore 
$$ \alpha(t)=\frac{\left[N(p-1)\gamma_{N}r^{N-2}t+(\gamma_{N}kr^{N-1}+1)^{1-p}\right]^{\frac{1}{1-p}}-1}{\gamma_{N}r^{N-1}}.$$
In what follows, we check that 
$$u(x,t) =\frac{\left(\left[N(p-1)\gamma_{N}r^{N-2}t+(\gamma_{N} kr^{N-1}+1)^{1-p}\right]^{\frac{1}{1-p}}-1\right)}{\gamma_{N}r^{N-1}}\chi_{B_{r}(0)}(x)\chi_{[0,T]}(t)$$
where $T= \frac{1-(\gamma_{N} kr^{N-1}+1)^{1-p}}{N(p-1)\gamma_{N}r^{N-2}}$, satisfies (\ref{Eq1})-(\ref{Eq3}). Since $u(t,x)=0$ on $\partial \Omega$, it easy to check that (\ref{Eq3}) holds. On the other hand, if $\varphi \in \mathcal{D}(\Omega)$ and $0\leq t\leq T$, we have 
\begin{multline*}
h(t) \int_{\Omega} \text{div}(z(t))\varphi \,dx=-h(t) \int_{\Omega} z(t).\nabla\varphi \,dx=-h(t) \int_{B_{r}(0)} z(t).\nabla\varphi \,dx-h(t) \int_{\Omega \backslash B_{r}(0)} z(t).\nabla\varphi \,dx\\
= -\frac{Nh(t)}{r}\int_{B_{r}(0)}\varphi \,dx+h(t)\int_{\partial B_{r}(0)}\frac{|x|^{2}}{r^{2}}\varphi\,d\mathcal{H}^{N-1}-h(t)\int_{\Omega\backslash B_{r}(0)}\text{div}\left(\frac{r^{N-1}x}{|x|^{N}}\right) \varphi \,dx\\
-h(t)\int_{\partial B_{r}(0)}\frac{r^{N-1}}{r^{N}}\frac{|x|^{2}}{r}\varphi\,d\mathcal{H}^{N-1}.
\end{multline*}
where $h(t)=\left[N(p-1)\gamma_{N}r^{N-2}t+(\gamma_{N} kr^{N-1}+1)^{1-p}\right]^{\frac{p}{1-p}}$.  Hence 
$$ m\left(\int_{ \Omega} |Du(t)|\right)\int_{ \Omega} \text{div}(z(t))\varphi \,dx=\int_{ \Omega} u'(t)\varphi \,dx, $$ 
and consequently (\ref{Eq1})  holds. Finally, by Green's formula (\ref{Gree}), we have 
\begin{eqnarray*}
	\int_{\Omega} (z(t), Du(t))&=&-\int_{\Omega} \text{div}(z(t))u(t)\,dx+\int_{\partial\Omega} [z(t), \nu]u(t)\,d\mathcal{H}^{N-1}\\
	&=& -\int_{ B_{r}(0)} \frac{\left(\left[N(p-1)\gamma_{N}r^{N-2}t+(\gamma_{N} kr^{N-1}+1)^{1-p}\right]^{\frac{1}{1-p}}-1\right)}{\gamma_{N}r^{N-1}}\text{div}(z(t))\,dx\\
	&=&\frac{N}{r}\int_{B_{r}(0)} \frac{\left(\left[N(p-1)\gamma_{N}r^{N-2}t+(\gamma_{N} kr^{N-1}+1)^{1-p}\right]^{\frac{1}{1-p}}-1\right)}{\gamma_{N}r^{N-1}}\,dx\\
	&=& \frac{\left(\left[N(p-1)\gamma_{N}r^{N-2}t+(\gamma_{N} kr^{N-1}+1)^{1-p}\right]^{\frac{1}{1-p}}-1\right)}{\gamma_{N}r^{N-1}} \frac{N}{r}|B_{r}(0)|\\
	&=&	\frac{\left(\left[N(p-1)\gamma_{N}r^{N-2}t+(\gamma_{N} kr^{N-1}+1)^{1-p}\right]^{\frac{1}{1-p}}-1\right)}{\gamma_{N}r^{N-1}} \mathcal{H}^{N-1}(\partial B_{r}(0))=\int_{\Omega} |Du(t)|.
\end{eqnarray*}
Therefore (\ref{Eq2}) holds, and consequently $u(x,t) $ is a solution of (\ref{E1}) with initial datum $u_{0}$. In a similar fashion one can show that (\ref{Sl2}) satisfies (\ref{Eq1})-(\ref{Eq3}). Hence the proof is now complete.

\section{proof of Theorem \ref{Th2} }
This section is concerned with the proof of Theorem \ref{Th2}. As we have pointed out before, the proof is based on a comparison  principle that we state in what follows. 

\begin{lemma}\label{TH2}
Let $u_{0}\in BV(\Omega)\cap L^{\infty}(\Omega)$ and let $u_{1}(x,t)$ be the unique solution of problem (\ref{E1}). Assume that $(m_{1})$ holds and let $d(\Omega)$ be the smallest radius of  a ball containing $\Omega$. Let $u_{2}(x,t)=\alpha (t)$, satisfying 
\begin{equation}\label{ET}
	|\alpha '(t)|\leq \frac{m(0)N}{d(\Omega)}.
\end{equation}
Then the following conclusion holds :
\begin{enumerate}

	\item If $\alpha(t)\geq 0$ and $u_{0}\leq \alpha(0)$, we have 
	$$ u_{1}(t)\leq u_{2}(t)\;\;\;\text{a.e. on}\; \Omega, $$
	\item If $\alpha(t)\leq 0$ and $u_{0}\geq \alpha(0)$, we have 
	$$ u_{1}(t)\geq u_{2}(t)\;\;\;\text{a.e. on}\; \Omega. $$
\end{enumerate}
\end{lemma}
\begin{proof}
Since $\Omega $ is bounded, without loss of generality, we may assume that $\Omega \subseteq B(0, d(\Omega))$. According to Theorem \ref{Th1} there exists $z_{1}(t)\in X(\Omega)$, $\|z_{1}(t)\|_{\infty}\leq 1$, satisfying

\begin{equation}\label{EQQ1}
	u'_{1}(t)-m\left(\|u_{1}\|\right)\text{div}(z_{1}(t))=0, \;\;\text{in}\;\mathcal{D}'(\Omega) \;\;\text{a.e}\;t\in [0,T]
\end{equation}
\begin{equation}\label{EQQ2}
 \int_{\Omega} (z_{1}(t), Du_{1}(t))	=\int_{\Omega}|Du_{1}(t)|, 
\end{equation}
\begin{equation}\label{EQQ3}
[z_{1}(t), \nu]\in \text{sign}(-u_{1}(t))\;\;\mathcal{H}^{N-1}-\text{a.e}\;\;\text{on}\;\partial \Omega.	
\end{equation}
In view of (\ref{ET}), it is easy see that $u_{2}(t)=\alpha(t)$ satisfies (\ref{EQQ1})-(\ref{EQQ3})
with the chosen  vector field $z_{2}(t)(x)=\frac{\alpha'(t)x}{N m(0)}$. Thus, by the Green's formula (\ref{Gree}), we get 
\begin{eqnarray*}
\frac{1}{2}\frac{d}{dt}\int_{\Omega} [(u_{1}(t)-u_{2}(t))^{+}]^{2}\,dx&=&-m\left(\|u_{1}\|\right) \int_{ \Omega} (z_{1}(t), D[(u_{1}(t)-u_{2}(t))^{+}])\\
&&+m(0)\int_{ \Omega} (z_{2}(t), D[(u_{1}(t)-u_{2}(t))^{+}])\\
&&+m\left(\|u_{1}\|\right) \int_{\partial\Omega} [z_{1}(t), v](u_{1}(t)-u_{2}(t))^{+}\,d\mathcal{H}^{N-1}\\
&&-m(0) \int_{\partial\Omega} [z_{2}(t),\nu](u_{1}(t)-u_{2}(t))^{+}\,d\mathcal{H}^{N-1}.
\end{eqnarray*}

If $R_{t}(r)=(r-\alpha (t))^{+}$, then  by using similar calculations as in the proof of \cite[Theorem 4]{Andreu}, we arrive at 
\begin{equation}\label{EWW}
\int_{\Omega}(z_{1}(t), DR_{t}(u_{1}(t)))= \int_{\Omega}|DR_{t}(u_{1}(t))|.
\end{equation}
Moreover, by (\ref{BR}) we deduce that 
\begin{equation}\label{EST}
 \left|\int_{\Omega} (z_{2}(t), R_{t}(u_{1}(t)))\right|\leq \|z_{2}(t)\|_{\infty}\int_{\Omega} |DR_{t}(u_{1}(t))|\leq \int_{\Omega} |DR_{t}(u_{1}(t))|.
\end{equation}
Combining (\ref{EWW}) and (\ref{EST}), and using the fact that $m$ is increasing function, we obtain
\begin{multline}\label{EA}
-m\left(\|u_{1}\|\right) \int_{ \Omega} (z_{1}(t), D[(u_{1}(t)-u_{2}(t))^{+}])+m(0)\int_{ \Omega} (z_{2}(t), D[(u_{1}(t)-u_{2}(t))^{+}])\\
\leq \left(m(0)-m\left(\|u_{1}\|\right)\right)\int_{\Omega}|DR_{t}(u_{1}(t))|	\leq 0.
\end{multline}
In light of $|[z_{2}(t), \nu]|\leq 1$, $[z_{1}(t), \nu]\in \text{sign}(-u_{1}(t))$ and $u_{2}(t)\geq 0$, we derive
\begin{multline}\label{RT1}
m\left(\|u_{1}\|\right) \int_{\partial\Omega} [z_{1}(t), \nu](u_{1}(t)-u_{2}(t))^{+}\,d\mathcal{H}^{N-1}-m(0) \int_{\partial\Omega} [z_{2}(t),\nu](u_{1}(t)-u_{2}(t))^{+}\,d\mathcal{H}^{N-1}\\
\leq\left(m(0)-m\left(\|u_{1}\|\right) \right)\int_{\partial\Omega \cap \{u_{1}>u_{2}\}} (u_{1}(t)-u_{2}(t))\,d\mathcal{H}^{N-1}\leq 0.
\end{multline}
Gathering (\ref{EA}) and (\ref{RT1}) yields
$$ \frac{1}{2}\frac{d}{dt}\int_{\Omega}[(u_{1}(t)-u_{2}(t))^{+}]^{2}\,dx\leq 0.$$
Hence the condition $u_{1}(0)\leq u_{2}(0)$ ensures $u_{1}\leq u_{2}$. The proof of $(2)$ is quite similar to $(1)$, so here we omit it. 
\end{proof}

\begin{proof} (Theorem \ref{Th2})
Take 
$$ \alpha(t):= \frac{N m(0)}{d(\Omega)}\left(\frac{ d(\Omega)\|u_{0}\|_{\infty}}{m(0)N}-t\right)^{+}. $$
It follows that 
$$ |\alpha'(t)|= \frac{N m(0)}{d(\Omega)} \quad \text{and}\quad \alpha(0)=\|u_{0}\|_{\infty}.$$ 
According to Lemma \ref{TH2}, we conclude

$$ -\alpha(t)\leq u(t)\leq \alpha(t), $$ 
and from which we obtain that (\ref{BO}) holds. Hence the proof now is complete. 
\end{proof}
\section{Proof of Theorem \ref{TH3} }
In this section we are concerned with the proof of Theorem \ref{TH3}. Here we only prove the inequality in (\ref{IN1}). The proof of the  inequality in (\ref{IN2}) is similar, so here we omit it. Since $\Omega$ is a bounded domain, then without lost of generality we may assume that $\Omega\subseteq B_{r} (0)$, for some $r>0$ . By Theorem \ref{Th1}, we know that 
	$$ v(t,x)=\frac{\left(\left[N(p-1)\gamma_{N}r^{N-2}t+(\gamma_{N} kr^{N-1}+1)^{1-p}\right]^{\frac{1}{1-p}}-1\right)^{+}}{\gamma_{N}r^{N-1}}\chi_{B_{r}(0)}$$  is a solution of problem (\ref{E1})  with the chosen vector field (\ref{VC})  and initial datum $u_{0}=k \chi_{B_{r}(0)}$. In order to proof (\ref{IN1}) we argue by contradiction by assuming that, there exists $t_{0}\in (0, T^{*}_{1}(u_{0}))$ such that 
	\begin{equation*}\label{IN3}
	\|u(t_{0})\|_{\infty} <\|v(t_{0})\|_{\infty}, 
	\end{equation*}
which implies that there exists $\epsilon > 0$ such that 
	\begin{equation}\label{IN4}
	\|u(t_{0})\|_{\infty} <\frac{\left(\left[N(p-1)\gamma_{N}r^{N-2}t_{0}+(\gamma_{N} kr^{N-1}+1)^{1-p}\right]^{\frac{1}{1-p}}-1\right)^{+}}{\gamma_{N}r^{N-1}}-\epsilon=k_{1}. 
	\end{equation}
	Now let us consider the following functions:  
	\begin{equation}\label{SOL1}
	v_{1}(x,t):=\frac{\left(\left[N(p-1)\gamma_{N}r^{N-2}t+(\gamma_{N} k_{1}r^{N-1}+1)^{1-p}\right]^{\frac{1}{1-p}}-1\right)^{+}}{\gamma_{N}r^{N-1}}\chi_{B_{r}(0)},
	\end{equation}
	and
	\begin{equation}\label{SOL2}
	v_{2}(x,t):=-\frac{\left(\left[N(p-1)\gamma_{N}r^{N-2}t+(\gamma_{N} k_{1}r^{N-1}+1)^{1-p}\right]^{\frac{1}{1-p}}-1\right)^{+}}{\gamma_{N}r^{N-1}}\chi_{B_{r}(0)}.	
	\end{equation}	
	In view of (\ref{IN4}), clearly we have that $v_{2}(0)\leq u(t_{0}) \leq v_{1}(0)$. Moreover,  proceeding similarly as in the proof of Theorem \ref{Th1}, one can show that $v_{1}$ and $v_{2}$ are solutions to (\ref{E1}) with the candidate vector field (\ref{VC})  and initial datum $u_{0}=k_{1} \chi_{B_{r}(0)}$ and $u_{0}=-k_{1} \chi_{B_{r}(0)}$ respectively. Thus by using the comparison principle  in Lemma \ref{Th2}, we deduce $v_{2}(t)\leq u(t_{0}+t)\leq v_{1}(t).$ Hence, it follows that
	\begin{eqnarray*}
 T^{*}_{1}(u_{0})-t_{0}&=&T^{*}_{1}(u(t_{0}))\leq \frac{1-(\gamma_{N} k_{1}r^{N-1}+1)^{1-p}}{N(p-1)\gamma_{N}r^{N-2}}\\
 &= & \frac{1-\left(\left[N(p-1)\gamma_{N}r^{N-2}t_{0}+(\gamma_{N} kr^{N-1}+1)^{1-p}\right]^{\frac{1}{1-p}}-\epsilon \gamma_{N}r^{N-1}\right)^{1-p}}{N(p-1)\gamma_{N}r^{N-2}}\\
 &< & T^{*}_{1}(u_{0})-t_{0}.
\end{eqnarray*}
But this is a contradiction. This completes the proof of the first statement. Now we turn to show the second statement of Theorem \ref{TH3}. Let $\zeta=\|u_{0}\|_{\infty}$, according to Theorem \ref{Th1} one can show that the following functions 
	
		\begin{equation}\label{SOL33}
	v_{1}(x,t):=\frac{\left(\left[N(p-1)\gamma_{N}r^{N-2}t+(\gamma_{N} \zeta r^{N-1}+1)^{1-p}\right]^{\frac{1}{1-p}}-1\right)^{+}}{\gamma_{N}r^{N-1}}\chi_{B_{r}(0)},
	\end{equation}
	and
	\begin{equation}\label{SOL4}
	v_{2}(x,t):=-\frac{\left(\left[N(p-1)\gamma_{N}r^{N-2}t+(\gamma_{N} \zeta r^{N-1}+1)^{1-p}\right]^{\frac{1}{1-p}}-1\right)^{+}}{\gamma_{N}r^{N-1}}\chi_{B_{r}(0)}.	
	\end{equation}	
	are solutions of (\ref{E1}) with initial datum $\zeta\chi_{B_{r}(0)}$ and $-\zeta\chi_{B_{r}(0)}$ respectively. Thus, by the comparison principle seen in (Lemma \ref{Th2}) , we have $v_{2} (x,t)\leq u(x,t)\leq v_{1}(x,t)$ for all $t\geq 0$ and $x\in \Omega$. Therefore, $supp (u(t))\subset B_{r}(0)$ for all $t\geq 0$. Hence, this ends the proof.
\section{Proof of Theorem \ref{Th3}}
This section is devoted to the proof of Theorem \ref{Th3}. Next we prepare the following lemma. This will be employed to prove Theorem \ref{Th3}. 
\begin{lemma}\label{EWS}
Let $u(t)$ be a solution of the Dirichlet problem (\ref{E1}). Suppose that (\ref{CN}) holds. Then there holds 
\begin{equation}
	|u'(t)|\leq \frac{|u_{0}|m\left( \frac{1}{\mu m_{0}}M\left(\int_{\Omega} |Du_{0}|+\int_{\partial \Omega} |u_{0}|\,d\mathcal{H}^{N-1}|u_{0}|\right)\right)}{t},
\end{equation}
for almost  all $t> 0.$
\end{lemma}
\begin{proof}
From (\ref{CN}), (\ref{EN}) and $(m_{1})$ we can easily obtain that 
\begin{equation}\label{EQ0}
	m\left(\int_{ \Omega} |Du(t)|\right)\leq m\left( \frac{1}{\mu m_{0}}M\left(\int_{\Omega} |Du_{0}|+\int_{\partial \Omega} |u_{0}|\,d\mathcal{H}^{N-1}\right)\right)
\end{equation}
Remembering that the solution (\ref{E1}) in the form $u(t) =v(\alpha(t))$ where $v$ is a solution of (\ref{E3}) and by \cite[Lemma 2]{Andreu2}, we have that the solution $v$ of (\ref{E2}) satisfies 
\begin{equation}\label{EQ1}
|v'(t)|\leq \frac{|u_{0}|}{t} \;\;\text{for almost  all}\; t> 0.
\end{equation}
Gathering (\ref{EQ0}) and (\ref{EQ1}), yields 
$$ |u'(t)|=\alpha'(t)|v'(t)|=	m\left(\int_{ \Omega} |Du(t)|\right) |v'(\alpha(t))|\leq \frac{|u_{0}|m\left( \frac{1}{\mu m_{0}}M\left(\int_{\Omega} |Du_{0}|+\int_{\partial \Omega} |u_{0}|\,d\mathcal{H}^{N-1}|u_{0}|\right)\right)}{t},$$ 
for almost  all $t> 0.$
\end{proof} 
Now let us turn to prove Theorem \ref{Th3}. By Theorem \ref{Th1} there exists $z(t) \in  X(\Omega)$, $\|z(t)\|_{\infty}\leq 1$, satisfying that 
\begin{equation}\label{EQQ11}
u'(t)-m\left(\int_{ \Omega} |Du(t)|\right)\text{div}(z(t))=0, \;\;\text{in}\;\mathcal{D}'(\Omega) \;\;\text{a.e}\;t\in [0,T]
\end{equation}
\begin{equation}\label{EQQ21}
\int_{\Omega} (z(t), Du(t))	=\int_{\Omega}|Du(t)|, 
\end{equation}
\begin{equation}\label{EQQ31}
[z(t), \nu]\in \text{sign}(-u(t))\;\;\mathcal{H}^{N-1}-\text{a.e}\;\;\text{on}\;\partial \Omega.	
\end{equation}
Multiplying (\ref{EQQ11})  by $w\in BV(\Omega)\cap L^{2}(\Omega)$, afterward integrating over $\Omega$ and using the Green's formula (\ref{Gree}), we get
\begin{equation}\label{AZ}
\int_{ \Omega} u'(t)w\,dx+m\left(\int_{ \Omega} |Du(t)|\right)\int_{ \Omega} (z, Dw)= \int_{\partial \Omega}[z(t), \nu]w\,d\mathcal{H}^{N-1},
\end{equation} 
 for every $w\in L^{2}(\Omega)\cap BV(\Omega)$.  Let $q\geq 1$, and $\varphi (r)=|r|^{q-1}r$. In view of (\ref{BO}) and \cite[Theorem 3.99]{Ambrosio}, we have that $\varphi (u)\in BV(\Omega)\cap L^{2}(\Omega)$. Then, taking $w=\varphi (u)$ as a test function in (\ref{AZ}), it follows that 
 \begin{equation}\label{AZ1}
 \int_{ \Omega} u'(t)\varphi (u)\,dx+m\left(\int_{ \Omega} |Du(t)|\right)\int_{ \Omega} (z, D\varphi (u))= \int_{\partial \Omega}[z(t), \nu]\varphi (u)\,d\mathcal{H}^{N-1},
 \end{equation} 
 Now, by \cite[Proposition 2.8]{Anz} and having in mind (\ref{EQQ21}), we have 
 $$ \int_{ \Omega} (z, D\varphi (u) )=\int_{ \Omega} \theta(z(t), D\varphi u(t), x)|D\varphi(u(t))=\int_{ \Omega} |D\varphi (u(t))|.$$ 
 Moreover, by (\ref{EQQ31})
 $$ [z(t), \nu]\varphi(u(t))=-|u(t)|^{q}\;\;\;\;\mathcal{H}^{N-1}-\text{a.e. on}\;\partial \Omega.$$
 Consequently, we get 
 \begin{equation}
 \frac{1}{q+1}\frac{d}{dt}\int_{ \Omega} |u(t)|^{q+1}\,dx+m_{0}\int_{ \Omega} |D\varphi (u(t))|+\int_{\partial\Omega} |u|^{q}\,d\mathcal{H}^{N-1}\leq 0. 
 \end{equation}
 By the continuous embedding (\ref{EMD}) there exists $\eta>0$ such that 
\begin{equation*}
 \frac{1}{q+1}\frac{d}{dt}\int_{ \Omega} |u(t)|^{q+1}\,dx+\eta\min\{m_{0}, 1\}\||u|^{q}\|_{\frac{N}{N-1}}\leq 0. 
 \end{equation*}
 Then, taking $q=N-1$ yields 
 \begin{equation}
\frac{d}{dt}\int_{ \Omega} |u(t)|^{N}\,dx+\eta \min\{m_{0}, 1\}N\left( \int_{ \Omega} |u(t)|^{N}\,dx\right)^{N-1/N}\leq 0. 
 \end{equation}
 Hence 
\begin{equation}\label{EC}
\frac{d}{dt}\left( \int_{ \Omega} |u(t)|^{N}\right)^{1/N}+\eta \min\{m_{0}, 1\}N\leq 0. 
\end{equation} 
Since $u(T^{*}(u_{0}))=0$, then by integrating (\ref{EC}) from $t$ to $T^{*} (u_{0})$, we obtain  (\ref{INQ}). This shows the first statement. For the second statement, by Lemma \ref{EWS} and  $u(T^{*}(u_{0}))=0$, for every $t\geq \tau$ we have 
\begin{eqnarray*}
 \left|\frac{u(x,t)}{T^{*}(u_{0})-t}\right|&=&\frac{|u(x,T^{*}(u_{0}))-u(x,t)|}{T^{*}(u_{0})-t}=\frac{1}{T^{*}(u_{0})-t}\left|\int_{t}^{T^{*}(u_{0})}u'(s)\,ds\right|\\	
 &\leq &\frac{1}{T^{*}(u_{0})-t}\int_{t}^{T^{*}(u_{0})}\frac{|u_{0}|m\left( \frac{1}{\mu m_{0}}M\left(\int_{\Omega} |Du_{0}|+\int_{\partial \Omega} |u_{0}|\,d\mathcal{H}^{N-1}|u_{0}|\right)\right)}{s}\,ds\\
 &\leq & \frac{|u_{0}|m\left( \frac{1}{\mu m_{0}}M\left(\int_{\Omega} |Du_{0}|+\int_{\partial \Omega} |u_{0}|\,d\mathcal{H}^{N-1}|u_{0}|\right)\right)}{\tau}.
\end{eqnarray*}
Hence this ends the proof of Theorem (\ref{Th3}).

\end{document}